\documentclass{amsart}
\usepackage{amsmath}
\usepackage{graphicx}
\usepackage{latexsym}

\newtheorem{thm}{Theorem}
\newtheorem{lem}[thm]{Lemma}
\newtheorem{cor}[thm]{Corollary}

\newtheorem{theorem}[thm]{Theorem}

\theoremstyle{definition}

\newtheorem{rk}[thm]{Remark}

\newtheorem{example}[thm]{Example}

\newcommand{\Z}{\mathbb{Z}}

\def \x {\times}
\def \eu{{\text{e}}}

\begin{document}

\title[Non-holomorphic surface bundles and Lefschetz fibrations]
{Non-holomorphic surface bundles \\ and Lefschetz fibrations}

\author[R. \.{I}. Baykur]{R. \.{I}nan\c{c} Baykur}
\address{Max Planck Institute for Mathematics, Bonn, Germany \newline
\indent Department of Mathematics, Brandeis University, Waltham MA, USA}
\email{baykur@mpim-bonn.mpg.de, baykur@brandeis.edu}

\begin{abstract}
We show how certain stabilizations produce infinitely many closed oriented $4$-manifolds which are the total spaces of genus $g$ surface bundles (resp.\,Lefschetz fibrations) over genus $h$ surfaces and have non-zero signature, but do not admit complex structures with either orientations, for ``most'' (resp.\,all) possible values of $g \geq 3$ and $h \geq 2$ (resp. $g \geq 2$ and $h \geq 0$).
\end{abstract}

\maketitle
\setcounter{secnumdepth}{2}
\setcounter{section}{0}

\section{Introduction}
Thurston and Gompf respectively showed that a closed oriented $4$-manifold which is the total space of a genus $g$ surface bundle or a Lefschetz fibration\footnote{Throughout this note, all manifolds are smooth, and all Lefschetz fibrations are assumed to have non-empty critical locus so as to make a clear distinction from surface bundles.} over a genus $h$ surface always admits a symplectic structure so that the fibration becomes symplectic, provided the regular fiber is homologically essential ---which might only fail for some genus one surface bundles. A natural question to ask is for which values of $g$ and $h$ the same holds when one replaces ``symplectic'' by ``complex'' above. The purpose of this note is to present answers to this question for both surface bundles and Lefschetz fibrations. 

As shown by Kotschick in \cite{K2} and independently by Hillman in \cite{Hil}, for $g, h \geq 2$, if the total space of a genus $g$ bundle over a genus $h$ surface admits a complex structure, then the bundle map can be homotoped to a holomorphic one for appropriate choices of complex structures on the total space and the base. Given such a holomorphic bundle $X \to B$, one gets a classifying map $\phi: B \to \mathcal{M}_g$, where $\mathcal{M}_g$ is the moduli space of non-singular genus $g$ curves. Using the Atiyah-Singer index theorem for families, one can then express $\sigma(X)$, the signature of the $4$-manifold $X$, as the evaluation on $B$ of the pull-back of the first Chern class of the Hodge bundle on $\mathcal{M}_g$ by $\phi$ \cite{BD, Atiyah, Smith}. It follows that the signature $\sigma(X) \neq 0$ when $g \geq 3$, unless the complex structure on $X$ is isotrivial \cite{BD}. For any $g, h \geq 1$, it is indeed easy to construct surface bundles with zero signature whose total spaces cannot be isotrivial or even complex; see Example \ref{Elementary} below. So for surface bundles, one shall reformulate the above question by adding the assumption on the non-vanishing of the signature, when $g \geq 3$ and $h \geq 2$.\footnote{We are therefore addressing a more extensive version of the MathOverflow question of Jim Bryan's, who asked whether every surface bundle could be made holomorphic, provided the signature of the total space is non-zero. Bryan later argued that the answer to his question should be ``not all' since otherwise one arrives a contradiction by taking fiber sums with trivial bundles over surfaces of arbitrarily big genera and then looking at the induced maps from these bundles into the moduli space $\mathcal{M}_g$. In this note we present topological constructions of such counter-examples for ``most'' pairs of positive integers $g,h$ and up to homotopies of such bundle maps.} It is well-known that the signature vanishes in the remaining low genera cases. 

In the next section, we will show how to obtain infinitely many (pairwise non-homotopic) closed orientable $4$-manifolds which are the total spaces of genus $g$ surface bundles over genus $h$ surfaces and have non-zero signatures, but do not admit complex structures with either orientations, for ``most'' possible values of $g$ and $h$, via stabilizations with  some elementary surface bundles (Theorem \ref{maintrick}). We prove that for all $g \geq 4$ and $h \geq 9$ such families exist (Corollary \ref{noncomplexbundles1}), and provide a weaker result for smaller base genera: For any positive integer $N$ and for $3 \leq h \leq 8$, we show that there are such families which are the total spaces of genus $g$ surface bundles over genus $h$ surfaces for some $g \geq N$ (Corollary \ref{noncomplexbundles2}). Constructions of these families rely on the work of Endo, Korkmaz, Kotschick, Ozbagci, and Stipsicz \cite{EKKOS} and Bryan and Donagi \cite{BD}, respectively.

There are many examples of torus bundles over tori, whose total spaces do not admit complex structures. (A comprehensive list can be found in \cite{Geiges}.) When the fibration has nodal singularities however, from Matsumoto's  classification result in \cite{Mat}, it is easy to conclude that any genus one Lefschetz fibration is smoothly isomorphic to a holomorphic elliptic Lefschetz fibration. Thus, one might wonder if the situation is more rigid for Lefschetz fibrations, at least when the base genus is positive. We show that this is not true either, by extending earlier results of Korkmaz given for the base genus zero case, and once again employing elementary stabilizations (Corollary \ref{noncomplexLFs}).

We can summarize our main results spelled out above as follows:

\begin{thm} \label{main} There are infinite families of (pairwise non-homotopic) closed \linebreak oriented $4$-manifolds with non-zero signatures, which do not admit any complex structure with either orientation, such that for fixed pairs of integers $g$ and $h$, each $4$-manifold in the respective family is the total space of \\
\noindent \text{\emph (i)} a genus $g$ surface bundle over a genus $h$ surface, for $g \geq 4$ and $h \geq 9$, \\
\noindent \text{\emph (ii)} a genus $g$ surface bundle over a genus $h$ surface, for $g=(4n - 2)n^2 + 1$, $n \geq 2$, and $3 \leq h \leq 8$, or \\ 
\noindent \text{\emph (iii)} a genus $g$ Lefschetz fibration over a genus $h$ surface, for $g \geq 2$ and $h \geq 0$.
\end{thm}

Parshin and Arakelov's proofs of the Geometric Shafarevich Conjecture show that there are finitely many holomorphic fibrations with fixed fiber genus $g \geq 2$, base genus $h \geq 0$, and degeneracy \cite{P,A}. The families constructed in Example \ref{Elementary} and Corollary \ref{noncomplexLFs} demonstrate that a symplectic analogue of this conjecture fails to hold. In fact, the Parshin-Arakelov finiteness theorem, combined with the array of results and observations mentioned in the second paragraph, hands us an alternative way to prove that within the families we get in Corollaries \ref{noncomplexbundles1} and \ref{noncomplexbundles2} infinitely many members are non-complex, instead of employing the Lemma \ref{notcomplex} of this note, which merely relies on the Enrique-Kodaira classification of complex surfaces. Then the gap in the literature we are filling essentially comes down to our construction of infinite families of surface bundles with fixed fiber and base genera whose total spaces have non-zero signatures.

\vspace{0.3in}
\noindent \textit{Acknowledgements.} The author was partially supported by the NSF grant DMS-0906912. We would like to thank Hisaaki Endo, Jonathan Hillman and Andras Stipsicz for helpful comments on a preliminary version of this paper.

\newpage
\section{Constructions}

In what follows, we will make repeated use of the following lemma:
 
\begin{lem} \label{notcomplex} 
Let $X$ be a closed oriented $4$-manifold with $b_1(X)$ odd. If $X$ is minimal, $b_2(X) \neq 0$, and either the euler characteristic $\eu(X)$ or the signature $\sigma(X)$ is non-zero, then neither $X$ nor $\bar{X}$ ($X$ with reversed orientation) admits a complex structure. In particular, the result holds if $X$ admits a genus $g$ bundle (resp. relatively minimal Lefschetz fibration) over a genus $h$ surface with $g \geq 2$, $h \geq 2$ (resp. $h \geq 1$).  
\end{lem}

\begin{proof}
Assume $X$ is minimal and admits a complex structure. From the Enrique-Kodaira classification of complex surfaces, a complex surface with odd first betti number is either of type VII or elliptic. Since $b_2(X) \neq 0$, $X$ cannot be of type VII.  On the other hand, it was shown by Kodaira that minimal elliptic surfaces with odd betti number necessarily have zero euler characteristic (Theorem 1.38 in \cite{Fundamental}), and in turn, they all have zero signature. So $X$ cannot be complex if either $\eu(X)$ or $\sigma(X)$ is non-zero. The same line of arguments shows that $\bar{X}$ cannot be complex either. 

Now suppose $X$ admits a genus $g$ bundle over a genus $h$ surface, with $g, h \geq 2$. First of all, $X$ is acyclic, and in particular, minimal. As the fiber genus $g \geq 2$, the fiber is homologically essential, so $b_2(X) \neq 0$. Moreover, $\eu(X)= 4(g-1)(h-1) \neq 0$. So the result follows as before. Lastly, if $X$ admits a relatively minimal genus $g$ Lefschetz fibration over a genus $h$ surface, and $g \geq 2$, $h \geq 1$, then the relative minimality of the fibration implies that $X$ is minimal as shown by Stipsicz in \cite{S}, $b_2(X) \neq 0$, and $\eu(X)= 4(g-1)(h-1) +k \neq 0$, where $k$ is the number of critical points. So the same arguments apply. 
\end{proof}

The stabilizations featured in the current article are of two types: \textit{Horizontal} and \textit{vertical stabilizations}, which are respectively the standard fiber sums and section sums of given surface bundles or Lefschetz fibrations with some elementary surface bundles which we present in the next example. All these stabilizations are instances of the generalized fiber sum operation, and they can be performed symplectically. 

\begin{example}[Elementary blocks] \label{Elementary} 
Let $m, g, h$ be positive integers, and $a, b$ be two non-separating simple closed curves whose homology classes are primitive and distinct in $H_1(\Sigma_g)$. For simplicity, assume that  $g + h \geq 3$, and $a$ and $b$ do not intersect when $h=1$. Now, let $t_a$ and $t_b$ denote the positive Dehn twists along $a$ and $b$, respectively. 

Set $p: P=P(g,h)=\Sigma_g \x \Sigma_h \to \Sigma_h$ to be the trivial genus $g$ bundle prescribed by the projection onto the second component of $P$. Define the bundle \linebreak $q_m: Q_m=Q_m(g,h,a) \to \Sigma_h$ as the one whose monodromy representation consists of only one non-trivial factor, which is $[t_a^m, 1]$. Similarly, we define the bundle $r_m: R_m= R_m(g,h,a,b) \to \Sigma_h$ as the one whose monodromy representation consists of only one non-trivial factor $[t_a^n, t_b]$ if $h=1$, and otherwise the one with two non-trivial factors $[t_a^m, 1]$ and $[t_b, 1]$. Here we will suppress $g, h, a$ and $b$ or sometimes only $a$ and $b$ from the notation. Since the monodromies of $p, q_m$, and $r_m$ are all supported on $\Sigma_g$ minus a disk, they admit sections $s_P$, $s_{Q_m}$ and $s_{R_m}$ of self-intersection zero, respectively. From the monodromy representations, it is easy to compute the first homologies as
\[
H_1(P)= \Z^{2g+2h} \, , \ \ H_1(Q_m)= \Z^{2g+2h-1} \oplus \Z_m \ , \text{and} \  H_1(R_m)= \Z^{2g+2h-2} \oplus \Z_m \ ,
\]
and the signatures of all as zero.  

The bundles $Q_m$ are straightforward generalizations of the famous Kodaira-Thurston manifold, which was the first example of a symplectic but not K\"ahler manifold. Since $b_1(Q_m)=2g+2h-1 $ is odd, for $g, h \geq 2$, we see that $q_m: Q_m \to \Sigma_h$ are examples of surface bundles whose total spaces cannot be complex by Lemma~\ref{notcomplex}. It is worth noting that when $g=1$, these manifolds do admit complex structures making the bundle holomorphic.
\end{example}

We are now ready to spell out our main recipe to produce surface bundles with non-zero signatures, whose total spaces cannot be complex: 

\begin{theorem} \label{maintrick} 
\emph{\textbf{(1)}  [Horizontal stabilization]} If $y: Y \to \Sigma_h$ is a genus $g$ surface bundle (resp. Lefschetz fibration) such that the the fiber is primitive in $H_2(Y)$, $g \geq 2$, $h \geq 1$ (resp. $h \geq 0$), and $b_1(Y)> 2h$ is even, then the horizontal stabilization of $(Y,y)$ with the genus $g$ bundle $(Q_m(g,h'), q_m)$ results in a genus $g$ surface bundle (resp. Lefschetz fibration) $y'_m: Y'_m \to \Sigma_{h+h'}$, where $Y'_m$ has the same signature as $Y$ and does not admit a complex structure with either orientation. The manifolds $Y'_m$ and $Y'_n$ are not homotopy equivalent for $m \neq n$. \vspace{0.1cm}

\emph{\textbf{(2)} [Vertical stabilization]} If $z: Z \to \Sigma_h$ is a genus $g$ surface bundle (resp. Lefschetz fibration) admitting a section $s_Z$ of self-intersection zero, with $g \geq 1$, $h \geq 2$  \linebreak (resp. $h \geq 1)$, then the section sum\footnote{We shall note that ``the'' section sum operation is not well-defined up to isomorphisms of surface bundles, or said differently, when the bundle maps are fixed, the resulting bundle, even up to isomorphism, depends on the chosen gluing map. Nevertheless, we ignore this issue as the results we state throughout the article hold for any such gluing.} of $(Z, z, s_Z)$ with either $(Q_m(g',h), q_m, s_{Q_m})$ or $(R_m(g',h), r_m, s_{R_m})$ depending on whether $b_1(Z)$ is even or odd, results in a genus $g+g'$ bundle (resp. Lefschetz fibration) $z'_m : Z'_m \to \Sigma_h$, where $Z'_m$ has the same signature as $Z$ and does not admit a complex structure with either orientation. The manifolds $Z'_m$ and $Z'_n$ are not homotopy equivalent for $m \neq n$. 
\end{theorem}

\begin{proof}
Let $y'_m: Y'_m \to \Sigma_h$ be the result of the fiber sum of the surface bundle $(Y, y)$ with $(Q_m, q_m)$. From the Mayer-Vietoris sequence, we calculate $H_1(Y'_n)$ as the cokernel of the homomorphism
\[ 
i_Y \oplus i_{Q_m}: H_1(\Sigma_g) \to H_1(Y) \oplus H_1(Q_m) ,
\]
where $i_Y$ and $i_{Q_m}$ are the homomorphisms induced by the inclusions of regular fibers into $Y$ and ${Q_m}$ respectively. Note that the classes $i_Y(\Sigma_g)$ and $i_{Q_m}(\Sigma_g)$ are both primitive in $H_1(Y)$ and $H_1(Q_m)$, respectively, so there is no extra torsion term involved in this calculation. (However our arguments would work out the same way when there is a torsion term in the case of a non-primitive fiber class \cite{Hamilton}.) 

Passing to real coefficients, we can easily calculate $b_1(Y'_n)$ from the induced map $i_Y \oplus i_{Q_m}$ above as $b_1(Y'_n)= b_1(Y)+b_1({Q_m})-2h +d$, where $d$ is the dimension of the kernel. From the homotopy exact sequence of a fibration, we have the exact sequence  
\[ 
0=\pi_2(\Sigma_h) \rightarrow \pi_1(\Sigma_g) \stackrel{(i_Y)_{*}}{\rightarrow} \pi_1(Y) \stackrel{y_{*}}{\rightarrow} \pi_1(\Sigma_h) \rightarrow 0 \, ,
\]
which in turn gives us the following piece of the $5$-exact sequence in homology:

\begin{equation} \label{homologysequence} 
H_1(\Sigma_g)_{\pi_1(\Sigma_h)}= \frac{\pi_1(\Sigma_g)}{[\pi_1(Y), \pi_1(\Sigma_g)]}  \stackrel{(i_Y)_{*}}{\rightarrow} H_1(Y) \stackrel{f_{*}}{\rightarrow} H_1(\Sigma_h) \rightarrow 0 \, .
\end{equation}

Since $b_1(Y)> 2h$ by assumption, there should be a class in $H_1(\Sigma_g)$ the image of which under $i_Y$ is   non-torsion in $H_1(Y)$. Let $\hat{a}$ be a non-separating simple closed curve on $\Sigma_g$ representing the primitive root of this homology class mapped to a non-torsion (resp. torsion) element. Then set $a=\hat{a}$ in the construction of \linebreak $(Q_m = Q_m(g, h',a), q_m)$, and take the standard fiber sum with $(Y,y)$. We therefore get $d=0$, and $b_1(Y'_n)= b_1(Y)+b_1(Q_m)-2h$ is odd. Moreover, we deduce that $H_1(Y'_m )$ has a $\Z_m$ component coming from $H_1({Q_m})$.

Now let $z'_m: Z'_m \to \Sigma_h$ be the result of the section sum of $(Z, z, s_Z)$ with any one of the triples in the statement, which we denote by $(T, t, s_T)$. From the Mayer-Vietoris sequence, $H_1(Z'_m)$ is the cokernel of the homomorphism
\[ 
s_Z \oplus s_T : H_1(\Sigma_h) \to H_1(Z) \oplus H_1(T) .
\]
Note that both $s_Z(\Sigma_h)$ and $s_T(\Sigma_h)$ are primitive classes in $H_1(Z)$ and $H_1(T)$, respectively, so once again, there is no torsion term involved in this calculation. Moreover, each section map should be injective onto its image. No matter whether $T$ is $Q_m$ or $R_m$, we see that $H_1(Z'_m)$ would have a $\Z_m$ component, and $b_1(Z'_m)=b_1(Z)+b_1(T)-2h$. Therefore, the parity of $b_1(Z'_m)$ is the same as the parity of $b_1(Z)+b_1(T)$. So for the choice of $T$ as indicated in the statement, we get $b_1(Z'_m)$ odd. 

For all the bundles $(Y'_m, y'_m)$, $(Z'_m, z'_m)$ we have constructed above, the fiber and base genera are at least two, thus by Lemma \ref{notcomplex} the total spaces do not admit complex structures with either orientations. Since $\sigma(Q_m)=\sigma(R_m)=0$, by the Novikov additity $\sigma(Y'_m)=\sigma(Y)$ and $\sigma(Z'_m) = \sigma(Z)$. Varying $m$ we get pairwise non-homotopic $4$-manifolds in each family $\{(Y'_m, y'_m) \}_{m \in \Z^+}$, $\{(Z'_m, z'_m) \}_{m \in \Z^+}$.

Analogous statements for Lefschetz fibrations are proved mutadis mutandis.  
\end{proof}

\begin{rk}
Although not needed for the results that will follow in this note, we can extend the part (1) of our theorem to stabilize surface bundles $(Y,y)$ with $b_1(Y)$ odd to obtain infinite families of bundles over higher genera surfaces. In this case, $b_1(Y) \neq 2g+ 2h$, so there is also a class in $H_1(\Sigma_g)$ mapped to a torsion element in $H_1(Y)$ under $i_Y$, a primitive root of which can be represented by a non-separating simple closed curve $\hat{b}$. We can then use a slightly more general version of the bundles $(R_m = R_m(g, h',a,b), r_m)$ where $a$ and $b$ can be matched with $\hat{a}$ and $\hat{b}$ to obtain the desired result.
\end{rk}

Given any surface bundle as in the statement of Theorem \ref{maintrick} whose total space has non-zero signature, we can now produce the families we are interested in. The genus $3$ surface bundles over genus $h \geq 9$ surfaces constructed in \cite{EKKOS}, which admit sections of self-intersection zero and have total spaces with non-zero signatures, allow us to get: 

\begin{cor} \label{noncomplexbundles1}
For every pairs of integers $g \geq 4$ and $h \geq 9$, there are infinitely many (pairwisenon-homotopic) closed oriented $4$-manifolds with non-zero signature, each of which admits a genus $g$ surface bundle over a genus $h$ surface, and yet do not admit any complex structure with either orientation. 
\end{cor}

\begin{rk} \label{genusthree}
The genus $3$ surface bundles over surfaces of \cite{EKKOS}, nor their various modifications seem to meet the homology criterion we required in part (1) of Theorem \ref{maintrick}. To the best of our knowledge, every genus $3$ surface bundle over a genus $h$ surface whose total space has non-zero signature that appears in the literature was constructed from some explicit monodromy factorization, and all have $b_1 = 2h$. On the other hand, none of the complex examples arising from Atiyah, Kodaira, Hirzebruch type of branched covering constructions have fiber genus $3$, leaving the genus $3$ case of the main question of this paper as a curious one.  
\end{rk}

For lower base genera, we have the following weaker result, relying on the surface bundles constructed in \cite{BD}:

\begin{cor} \label{noncomplexbundles2}
For any positive integer $N$ and for $3 \leq h \leq 8$, there exists $g \geq N$ such that there are infinite families of (pairwise non-homotopic) closed oriented $4$-manifolds with non-zero signature, each of which admits a genus $g$ surface bundle over a genus $h$ surface, and yet do not admit any complex structure with either orientation. 
\end{cor}

\begin{proof}
In \cite{BD}, the authors construct complex surfaces $X_n$ with signature $\sigma(X_n)= \frac{8}{3}(n^3-n)$ admitting   two fibrations with fiber and base genera $(g,h)$ equal to \linebreak $((4n-2)n^2 + 1, 2)$ and $(2 n^2 + 1, 2n)$, respectively, for each $n \geq 2$. Applying the exact sequence (\ref{homologysequence}) above to the latter family of fibrations, we conclude that $b_1(X_n) \geq 2 n^2+ 1 > 4$. (Or one can see this from the authors' branched covering construction.) So we can horizontally stabilize the former fibrations to get the desired result. 
\end{proof}

\begin{rk} 
There are other surface bundles over small base genera (as small as $3$) constructed using similar branched covering techniques following Atiyah, Kodaira, and Hirzebruch (\cite{Hirzebruch, BDS, Lebrun}), to which the horizontal stabilizations of Theorem \ref{maintrick} can be applied to construct similar families of surface bundles with fiber genera different than those in the above proof.  
\end{rk}

We now turn to Lefschetz fibrations. In \cite{OS}, Ozbagci and Stipsicz constructed genus $2$ Lefschetz fibrations over the $2$-sphere, total spaces of which cannot admit complex structure with either orientation, using a twisted fiber sum of two copies of a genus $2$ Lefschetz fibrations over the $2$-sphere due to Matsumoto \cite{Mat2}. Matsumoto's fibrations, and in turn the construction of Ozbagci and Stipsicz were generalized by Korkmaz in \cite{Kork} for any $g \geq 2$. (Also see \cite{SmithDis} for other $g=2$, \cite{EN} for $g=3,4,5$ examples, and \cite{FS1, FS2} for families whose total spaces are in a fixed homeomorphism class for any $g \geq 3$.) Although these generalized Matsumoto fibrations and the standard fiber sum operation have complex interpretations, the twisted fiber sum operation does not, which allows one to perform it carefully so as to obtain fibrations on $4$-manifolds which cannot support complex structures.

We will present a simple extension of these constructions to any base genera: 

\begin{cor} \label{noncomplexLFs}
For every $g \geq 2$ and $h \geq 0$, there are infinitely many different (pairwise non-homotopic) closed oriented $4$-manifolds with non-zero signature, each of which admits a genus $g$ Lefschetz fibration with non-empty critical locus over a genus $h$ surface, and yet do not admit any complex structure with either orientation.
\end{cor}

\begin{proof}
Base genus $h=0$ case is covered by Korkmaz's examples. To apply Theorem~\ref{maintrick}, all we need is a relatively minimal genus $g$ Lefschetz fibration over the $2$-sphere whose total space has positive even first betti number. Such examples are found in abundance; we can for instance take the relatively minimal genus $g$ Lefschetz fibration $y: Y(g) \to S^2$ constructed by Korkmaz \cite{Kork} for any $g \geq 2$, where $H_1(Y)= \Z^{g}$ and $\sigma(Y)= -4$ for $g$ even, whereas for $g$ odd, $H_1(Y)= \Z^{g-1}$ and $\sigma(Y)=-8$. Hence the horizontal stabilizations of Theorem \ref{maintrick} yield the desired families.

Alternatively, by employing one more ---twisted--- fiber sum, we can take the relatively minimal genus $g$ Lefschetz fibrations $y_n: Y_n(g) \to S^2$ of Korkmaz's \cite{Kork}, which have $\pi_1(Y_n(g)) = \Z \oplus \Z_n$, for any $g \geq 2$ and $n \geq 1$. We then obtain the same result by performing horizontal stabilizations with $P(g,h)$ of Example \ref{Elementary} and invoking Lemma \ref{notcomplex} again.
\end{proof}

\end{document}